\documentclass{amsart}
\usepackage{etoolbox,xcolor,cancel}
\patchcmd{\section}{\scshape}{\bfseries}{}{}
\makeatletter
\renewcommand{\@secnumfont}{\bfseries}
\makeatother
\calclayout
\usepackage{enumitem}
\usepackage{amsmath,amssymb,amscd,amsthm,enumitem,yhmath,bbm}
\usepackage{graphicx,epstopdf}
\usepackage[titletoc,toc,title]{appendix}
\usepackage{color,soul}
\usepackage{float}

\def\e{\epsilon}

\def\bp{\begin{proposition}}
\def\ep{\end{proposition}}
\def\bt{\begin{theo}}
\def\et{\end{theo}}
\def\be{\begin{equation}}
\def\ee{\end{equation}}
\def\bl{\begin{lemma}}
\def\el{\end{lemma}}
\def\bc{\begin{corollary}}
\def\ec{\end{corollary}}

\def\bd{\begin{definition}}
\def\ed{\end{definition}}

\def\barP{{\cal \bar{P}}}

\def\1{{\mathbbm 1}}

\newcommand{\x}[1]{{}$\kern-2\mathsurround${}\binoppenalty10000 \relpenalty10000 #1{}$\kern-2\mathsurround${}}
\newcommand{\tab}{\hspace*{2em}}

\makeatletter
\newcommand*{\rom}[1]{\expandafter\@slowromancap\romannumeral #1@}
\makeatother

\makeatletter
\DeclareRobustCommand*\cal{\@fontswitch\relax\mathcal}
\makeatother

\theoremstyle{plain}
\newtheorem{theorem}{Theorem}[section]
\newtheorem{lemma}{Lemma}[section]
\newtheorem{definition}{Definition}[section]
\newtheorem{corollary}{Corollary}[section]
\newtheorem{proposition}{Proposition}[section]
\newtheorem{problem}{Problem}[section]
\newtheorem{remark}{Remark}[section]

\numberwithin{equation}{section}

\begin{document}

\title[]{On algebraic properties of low rank approximations of Prony systems}

\author[G. Goldman]{Gil Goldman}
\address{Department of Mathematics,     The Weizmann Institute of Science, Rehovot 76100, Israel}
\email{gilgoldm@gmail.com}
\thanks{}

\author[Y. Yomdin]{Yosef Yomdin}
\address{Department of Mathematics,     The Weizmann Institute of Science, Rehovot 76100, Israel}
\email{yosef.yomdin@weizmann.ac.il}

% \thanks{The work was supported by the RFBR grant 15-01-00745 A; ISF, Grant No. 779/13.}


\date{}

\begin{abstract}
	We consider the reconstruction of spike train signals of the form 
	$$F(x) = \sum_{i=1}^d a_i \delta(x-x_i),$$ 
	from their moments measurements $m_k(F)=\int x^k F(x) dx = \sum_{i=1}^d a_ix^k$. When some 
	of the nodes $x_i$ near collide the inversion becomes unstable. Given noisy moments measurements,
	a typical consequence is that reconstruction algorithms estimate the signal $F$ with a signal having fewer nodes, $\tilde{F}$.
	We derive lower bounds for the moments difference between a signal $F$ with $d$ nodes and a signal $\tilde{F}$ with strictly less nodes, $l$. 
	Next we consider the geometry of the non generic case
	of $d$ nodes signals $F$, for which there exists an $l<d$ nodes signal $\tilde{F}$, with moments 
	$$m_0(\tilde{F})=m_{0}(F),\ldots,m_{p}(\tilde{F})=m_{p}(F),\tab p>2l-1 .$$ 
	We give a complete description for the case of a general $d$, $l=1$ and $p=2$.           
	We give a reference for the case $p=2l-1$ which can be inferred from earlier work.
\end{abstract}

\maketitle

\section{Introduction}\label{Sec:Intro}
In this paper we consider the classical {\it Prony system} of algebraic equations
\begin{align}\label{eq.prony.system}
	\sum_{j=1}^d a_j x_j^k = m_k, && k= 0,1,\ldots,N,
\end{align}
with the unknowns $a_j,x_j, \
j=1,\ldots,d,$ and known right hand side formed by the moment ``measurements'' $m_0,\ldots,m_{N}$.
We will refer to the unknowns $a=(a_1,\ldots,a_d)$ as amplitudes and to the unknowns 
$x=(x_1,\ldots,x_d)$ as nodes.
    
\smallskip

Prony systems appear in many classical theoretical and applied mathematical problems
\cite{akinshin2015Rus,azais2015spike,batenkov2017accurate,condat2014new,batenkov2018accuracy,batenkov2017algebraic}.
% \marker{TODO: add ref to appearances singularity theory, Gauss quadratures and pade approximations.}
In particular, the bibliography in \cite{auton1981investigation} contains more than 50 pages.
Explicit solution of Problem (\ref{eq.prony.system}) was given by Prony himself already in \cite{prony1795essai}.

\smallskip

Many of the more recent applications are in Signal Processing. 
As a very partial sample we mention that in
\cite{blu2008sparse} and in many other publications a method, essentially equivalent to solving a Prony system, 
was used in reconstructing signals with a ``finite rate of innovation''. 
In \cite{peter2013generalized,plonka2014prony} the applicability of Prony-type systems was extended to 
some new wide and important classes of signals.
In \cite{bernardi2014comparison, comon2008symmetric} multidimensional Prony systems were investigated via symmetric
tensors, in particular, connecting them to the polynomial Waring problem. 
In \cite{eldar2015sampling} Prony system appears in the general context of Compressed Sensing. 
In \cite{batenkov2015complete, batenkov2012algebraic} Prony-like systems were used in reconstructing 
piecewise-smooth functions 
from their Fourier data. Finally, in \cite{batenkov2015complete} the same reconstruction accuracy as 
for smooth functions was demonstrated (thus confirming the Eckhoff conjecture).

\smallskip

In what follows we will identify the unknown tuple $(a,x)$ with a ``spike-train signal'' $F$,

\be\label{eq.delta.signal}
	F(x)=\sum_{j=1}^{d}a_{j}\delta(x-x_{j}).
\ee
Clearly, the moments $m_k(F)=\int x^k F(x)dx, \ k=0,1,\ldots,$ 
are given by $m_k(F)=\sum_{j=1}^d a_j x_j^k$, so reconstructing $F$ 
from its $N$ initial moments is equivalent to solving (\ref{eq.prony.system}), with $m_k=m_k(F)$.

\smallskip

In practice it is important to have \emph{a stable method
of inversion} and many research efforts are devoted to this task (see
e.g. \cite{badeau2008performance,
batenkov2013accuracy,donoho2006stable,peter2011nonlinear,potts2010parameter,stoica1989music} and references therein).
A basic question here is the following.
We are given \emph{noisy} measurements $\nu=(\nu_{0},\dots,\nu_{N})$ with
\begin{align}\label{eq.prony.system.noise}
	|\nu_k - m_k| \le \e,& & k= 0,1,\ldots,N,
\end{align}
where $m_k$ are actual moments for some signal $F$.
The goal is to solve the Prony system \eqref{eq.prony.system} with right hand side $\nu$, so as to minimize the worst case
reconstruction error.

\smallskip

An important case that poses major mathematical and numerical difficulties
is when some of the nodes $x_j$ of the measured signal nearly collide. 
In particular, this happens in the context of the ``super-resolution problem'', 
which was investigated in many recent publications. See
\cite{akinshin2017accuracy,batenkov2016stability,batenkov2013accuracy,candes2014towards,demanet2015recoverability,salman2018prony,plonka2017approximation} as a small sample.

% A basic question here is the following.
% \begin{problem}[Noisy Prony problem]\label{prob.noisy-prony}
% We are given the \emph{noisy} measurements $\mu'=(\mu_{0}',\dots,\mu'_{N})$ with 
% \begin{align}\label{eq.prony.system.noise}
% 	|\mu'_k - m_k| \le \e,& & k= 0,1,\ldots,N,
% \end{align}
% where $m_k$ are actual moments for some amplitudes and nodes $(A,X)$.
% The goal is solve the Prony system \eqref{eq.prony.system} with right hand side $\mu'$ so as to minimize worst case the
% reconstruction error. 
% % $$\min_{(A^*,X^*)}\max_{(A,X)} ||(A,X)-(A^*,X^*)||}$$
% \end{problem}

\smallskip

We now introduce the moments Hankel matrix which is important in the next calculations,
and is used in reconstruction algorithms that are based on Prony method.
Given a moments vector $m=\left(m_{0},\dots,m_{N}\right)$, with $N=2d-1,
d\in \mathbb{N}^+$, consider
the associated  $d\times d$ Hankel matrix $H_{d}(m)$,
\begin{equation}\label{eq.hankel.definition}
	H_{d}(m)=
	\begin{bmatrix}
		m_{0} & m_{1} & m_{2} & \dots & m_{d-1}\\
		m_{1} & m_{2} & m_{3} & \dots & m_{d}\\
		\adots & \adots & \adots & \adots & \adots\\
		m_{d-1} & m_{d} & m_{d+1} & \dots & m_{2d-2}
	\end{bmatrix}.
\end{equation}

We say that a signal $F$ as above, has $d$ nodes if its amplitudes $a_i,\ i=1,\ldots,d$ are non zero and 
the nodes are distinct.   
For exact measurements vector $m=(m_0,\ldots,m_N)$ of a signal $F$, 
% with $m_k=m_k(F),\;k=0,\ldots,N,$ for some signal $F$,
the rank of the associated Hankel matrix $H_{d}(m)$ is equal to number of nodes of $F$.   
Given noisy moment measurements $\nu=(\nu_{0},\dots,\nu_{N})$, generalized Prony methods for
reconstruction of $F$ typically estimate the numerical rank $r$, of the associated Hankel matrix $H_d(\nu)$.
The next step is to recover $F$ from $\nu$ with the number of nodes equal to $r$. 

As the nodes collide the rank of $H_d(\nu)$ drops, effectively causing such methods of reconstruction
to estimate the source signal with a signal with less nodes. Typically, 
each cluster of nodes will be reduced to a single node.

\smallskip

In the present paper we consider two problems related to a low rank
approximation of Prony systems.
% 
% A natural question to ask then is: given $N=2d-1,
% d\in \mathbb{N}^+$ and for each $r \le d$, what are the moment vectors $\mu=(m_0,\ldots,m_N)\in {\cal M}_d$ for which
% the Prony system \eqref{eq.prony.system} is exactly solvable. Here ${\cal M}_d$ denotes the space of moment vectors
% $(m_0,\ldots,m_N)$. The investigation of this question was started 
% in \cite{batenkov2013geometry}. It turns out to lead to a stratification of ${\cal M}_d$ according 
% to the rank of $H_{d}(\mu)$. 

% \smallskip
% 
% In this paper we present a related and somewhat ``dual'' question. 
Denote by ${\cal P}={\cal P}_d$ the parameter space of signals $F$ with $d$ nodes,
$$
	{\cal P}_d=\left\{(a,x)=(a_1,\ldots,a_d,x_1,\ldots,x_d)\in {\mathbb R}^{2d}, \ x_1<x_2<\ldots<x_d, 
	\ a_i \ne 0,\ i=1,\ldots,d \right\}.
$$
For the sack of completion we define ${\cal P}_0$ to be the singleton containing the zero signal $F_0(x)=0$.      
Denote by ${\cal P}^a_d$ and by ${\cal P}^x_d$ the parameter spaces of the amplitudes $a$ and the nodes $x$,
respectively.
Finally denote by ${\cal M}={\cal M}_{d} \cong {\mathbb R}^{2d}$ the moment space 
consisting of the $2d$-tuples of the form $(\nu_0,\nu_1,\allowbreak\ldots,\nu_{2d-1})$.

\smallskip

The first main question of this paper, considered in section \ref{sec.exact}, is the following:

\begin{problem}\label{prob.low.rank}
	Given the triplet $(d,l,p)$ of natural numbers with $d > l > 0 $, describe the geometry of the set of
	all signals $F \in {\cal P}_d$  such that there exists a signal $\tilde{F}$
	with at most $l$ nodes, $\tilde{F} \in {\cal P}_i, \; i\le l$, satisfying
	\begin{align*}
		m_k(\tilde{F})=m_k(F), & & k=0,\ldots,p.
	\end{align*}
	That is, $\tilde{F}$ matches the $p+1$ initial moments of $F$. 
\end{problem}
\noindent For each triplet $(d,l,p)$ as in Problem \ref{prob.low.rank}, we will denote by $\Sigma_{d,l,p} \subset {\cal P}_d$
the set of all signals satisfying the condition of the Problem for this case.   	 
Fixing $p$ and $d$ and then varying $l$ naturally leads to a stratification of ${\cal P}_d$ according to the sets $\Sigma_{d,l,p}$.
% 
% \smallskip
% 
% First we note that when $l=d$, $\Sigma_{d,l,p}=\Sigma_{d,d,p}$ is the whole space ${\cal P}_d$. Hence from now on we assume $l<d$.  

\smallskip

Let $\hat\Sigma_{d,l,p}\subset {\cal P}_d\times {\cal P}_l$ be defined by the algebraic conditions 
$$
	m_k(F)=m_k(\tilde F), \tab k=0,\ldots,p.
$$ 
Then for $\pi:{\cal P}_d\times {\cal P}_l\to {\cal P}_d$, the projection to the first factor, we have
$$
\Sigma_{d,l,p}=\pi(\hat\Sigma_{d,l,p}).
$$
In particular, this implies that $\Sigma_{d,l,p}$ is a semi-algebraic subset of ${\cal P}_d$. 
Counting the parameters, we can expect that for $p > 2l-1$ it is generically of codimension $p-2l+1$. 
For $p=2l-1$ and $F\in {\cal P}_d$ the condition $F\in \Sigma_{d,l,p}$ is equivalent to the solvability of the Prony system 
$m_k(\tilde F)=m_k(F), \ k=0,\ldots,p,$ for signals $\tilde F\in {\cal P}_l$ {\it with real nodes and amplitudes}. 
These conditions can be given explicitly (see, e.g. \cite{batenkov2013geometry,salman2018prony}).

\smallskip

Our main result with respect to Problem \ref{prob.low.rank} is a complete description of the geometry of the set of signals 
meeting the condition of the Problem for the case $(d,1,2)$. See Theorems \ref{thm.quad.form} and \ref{thm.linear.subspaces}.
% 
% \begin{remark}
% 	We note here that in terms of reconstruction algorithms and when some of the nodes of the measured signal $F$ nearly collide, the case $(d,1,p)$  
% 	concerns the possibility of estimating the signal with a single node. 
% \end{remark}      

\smallskip

The second main question, considered in section \ref{sec.lower.bound},  
is to provide lower bounds for the {\it errors in the moments} which appears as the consequence of approximating a signal $F$ with 
$d$ nodes by signals with at most $d-1$ nodes. 
We give a bound of this form in terms of the minors of the moment Hankel matrix \eqref{eq.hankel.definition} formed by the moments of $F$.
See Theorems \ref{thm.moment.dif} and Corollary \ref{cor.lower.bound}. 
Finally, as a special case, we consider a situation where the nodes of $F$ form a cluster of a size $h\ll 1$.
 
\section{Exact moment fitting}\label{sec.exact}
% The result is given in two steps. 
% First we show in Theorem \ref{thm.low.rank} that $F=(a,x) \in \Sigma_{d,1,2}$ iff exactly one of two mutually 
% exclusive conditions is satisfied. Either $m_0(F)=0$ and $$   
% 
% 
% The set $\Sigma_{d,1,2} \subset {\cal P}_d$ will be describe in Theorem \ref{thm.low.rank} as the union of three quite
% ``simple'' fiber bundles $P_1,P_2,Q\subset {\cal P}_d$. In each fiber bundle the nodes $x$ are used to parametrize the
% amplitudes, i.e. the base space in each bundle is the nodes space ${\cal P}^x_d$ and the fibers are subspaces of the
% amplitudes space ${\cal P}^a_d$.
% 
% 
% We will now construct these bundles starting with $Q$.
% 
% \smallskip
% 
% Let $e_1,\ldots,e_d \in {\mathbb R}^d$ be the standard coordinate vectors of ${\mathbb R}^d$ and let
% $E_i \subset {\mathbb R}^d,\; i=1,\ldots,d,$ be the $d-1$ dimensional linear subspaces
% with $E_i$ the subspace orthogonal to $e_i$. Finally let $E=\cup^d_{i=1}E_i$.
% 
% \smallskip
In this section we consider Problem \ref{prob.low.rank} of exact fitting of the moments of a signal $F$ by a signal with strictly less nodes $\tilde{F}$.

\smallskip 

For each signal $F=(a,x) \in {\cal P}_d$, consider the $d \times d$ matrix $D=D(x)$, its i,j entry is given by
$D_{i,j}  = d_{i,j}^2$, $d_{i,j}=x_i-x_j$.
The matrix $D$ is called an Euclidean Distance Matrix which has many important applications
(for applications in signal processing see (\cite{parhizkar2013euclidean}).   

\smallskip

In what follows we describe the geometry of the set of signals $\Sigma_{d,1,2}$.    
We show in Theorem \ref{thm.quad.form} that $F=(a,x) \in \Sigma_{d,1,2}$ iff 
the amplitudes vector of $F$, $a$, is a zero of the quadratic form induced by the Euclidean Distance 
Matrix supported by the nodes of $F$, $D(x)$.
This holds for all signals $F$ in $\Sigma_{d,1,2}$ with $m_0(F) \ne 0$.
For signals $F\in \Sigma_{d,1,2}$ with $m_0(F)=0,$ the description is straight forward and is given below as well. 

\smallskip

Next we use the spacial structure of $D(x)$ to show that for a fixed nodes vector $x$, the set of amplitudes vectors of
signals $F=(a,x) \in \Sigma_{d,1,2}$ having $m_0(F)\ne 0$, is a union of two sets $P_1(x),P_2(x)$. The sets
$P_1(x),P_2(x)$ are linear subspaces of dimension $d-1$ minus certain linear subspaces of dimension smaller than $d$.
This is done in Theorem \ref{thm.linear.subspaces}.

\smallskip

For a nodes vector $x= (x_1,x_2,\ldots,x_d) \in {\cal P}^x_d$,
denote by $V(x) = V_d(x)$ the Vandermonde matrix with infinite row index and $d$ columns and with the nodes $x_1,\ldots,x_d$,
% \begin{align*}
% 	V(x)_{i,j} &= x_j^i & i=0,\ldots,d-1,\; j=1,\ldots,d.
% \end{align*}
% 
\begin{equation}\label{eq.vandermonde}
	V_d(x)=
	\begin{bmatrix}
	1  & 1 &... & 1     \\
	x_{1}  & x_2 &... & x_{d}       \\
	x_1^{2}& x_2^{2}   & ...    & x_d^{2}\\
	\reflectbox{$\ddots$}  & \reflectbox{$\ddots$} &     &\reflectbox{$\ddots$}        
	\end{bmatrix}.
\end{equation}
We denote by $V^{0:k-1}(x) = V_d^{0:k-1}(x)$ the $k \times d$ submatrix of $V(x)$ formed by the first $k$ rows of $V(x)$.

\smallskip
\begin{theorem}\label{thm.quad.form}
	For $F=(a,x) \in {\cal P}_d$, $F \in \Sigma_{d,1,2}$ iff  $a_i \ne 0,\; i=1,\ldots,d$ and 
	exactly one of the following mutually exclusive conditions is met:
	\begin{enumerate}[label=(\Roman*)]
	  \item $m_0(F) = 0 $ and the amplitudes vector $a$ is in the null space of $V^{0:2}(x)$. 
	  \item $m_0(F) \ne 0 $ and the amplitudes vector $a$ is a zero of the quadratic form $a^TD(x)a$.    
	\end{enumerate}
\end{theorem}
The condition $a_i \ne 0,\; i=1,\ldots,d$ above is a mere technicality needed to ensure that $F$ has $d$ nodes.
\begin{proof}
	Let $F=(a,x) \in \Sigma_{d,1,2}$, $x=(x_1,\ldots,x_d),\;a=(a_1,\ldots,a_d)$. 
	If $m_0(F) =  0$, by assumption there exists $\tilde{F} \in {\cal P}_1$ or $\tilde{F} \in {\cal P}_0$ 
	such that $m_0(\tilde{F})=m_0(F) = 0$. Then $\tilde{F} = F_0 \in {\cal P}_{0}$, the identically 0 signal.
	Since $\tilde{F}$ has all its moments equal to $0$ we have that $m_0(F)=m_1(F)=m_{2}(F)=0$.
	The last condition is equivalent to the amplitudes vector of $F$, $a$, being in the null space of truncated Vandermonde $V^{0:2}(x)$. 
	
	\smallskip
	
	Else $m_0(F) \ne  0$.
	Then, $\tilde{F}$ is a single node signal with a non zero amplitude, $\tilde{F}(x)=\tilde{a}_1\delta(x-\tilde{x}_1),\
	\tilde{F}\in {\cal P}_1$.  
	\begin{lemma}\label{lemma.moments}
		Let $F\in {\cal P}_d,$ $F(x)=\sum_{i=1}^d a_i\delta(x-x_i)$, $d>1$.
		Let $\tilde{F}(x)=\tilde{a}_1\delta(x-\tilde{x}_1)$ be s single node signal such that $m_0(\tilde{F})=m_0(F)\ne 0$ and 
		$m_1(\tilde{F})= m_1(F)$.
		Then,
		\begin{equation}\label{eq.moment.diff}
				m_2(\tilde{F})-m_2(F) =  \frac{\sum_{i<j}a_ia_j(x_i-x_j)^2}{\sum_{i=1}^d a_i}= \frac{a^T D(x) a}{\sum_{i=1}^d a_i}.
		\end{equation} 
	\end{lemma}
		\begin{proof}
			$m_0(\tilde{F})=m_0(F)$ and $m(\tilde{F})_1 = m_1(F)$ imply that
			$\tilde{a}_1 = m_0(F)$ and  $\tilde{x}_1 = \frac{m_1(F)}{m_0(F)}$. 
			Then $m_2(\tilde{F}) = \tilde{a}_1\tilde{x}_1^2 = \frac{m_1^2(F)}{m_0(F)}$ and we have
			\begin{align*}
				m_2(F)-m_2(\tilde{F}) &= \frac{m_0(F) m_2(F)-m_1^2(F)}{m_0(F)}  \\
												 &=  \frac{(\sum_{i=1}^d a_i) (\sum_{i=1}^d a_i x_i^2)-(\sum_{i=1}^d a_ix_i)^2}{\sum_{i=1}^d a_i} \\
												 &= \frac{\sum_{i=1}^da_i^2x_i^2 +\sum_{1\le i < j \le d }(a_ia_jx_j^2 +a_ja_ix_i^2)
												 -\sum_{i=1}^d a_i^2x_i^2 -2 \sum_{1 \le i < j \le d} a_ia_jx_ix_j}{\sum_{i=1}^d a_i}\\
												 &= \frac{ \sum_{1 \le i < j \le d}a_ia_jx_j^2  - 2 a_ia_jx_ix_j +a_ia_jx_i^2}{\sum_{i=1}^d
												 a_i}\\
												 &=  \frac{ \sum_{1 \le i < j \le d}a_ia_j(x_i-x_j)^2}{\sum_{i=1}^d a_i} 
												 \allowdisplaybreaks .
			\end{align*}
	This conclude the proof of Lemma \ref{lemma.moments}. 			
	\end{proof}
	Case \rom{2} of Theorem \ref{thm.quad.form} now follows from Lemma \ref{lemma.moments} which concludes the proof of the Theorem.  
\end{proof}

We now consider case \rom{2} of Theorem \ref{thm.quad.form}. 
% 
% We analyse the zeros of the quadratic form $a^TD(x)a$
% to show that in this case the amplitudes vectors of signals  
% 
% , for a nodes vector $x$, we analyse the zeros of the quadratic form $a^TD(x)a$ to show it vanishes on the union of 
% two linear subspaces $P_1(x),P_2(x)$. The result is that for $F=(a,x) \in {\cal P}_d$ with $m_0(F) \ne 0$, 
% $F \in \Sigma_{d,1,2}$ with iff $a \in P_1(x)$ or $a \in P_2(x)$, provided that $a_i \ne 0, i=1,\ldots,d$.
% 
% We precede to define the bundles $P_1$ and $P_2$.
% 
% \smallskip

\smallskip

Introduce the maps $b_1,b_2: {\cal P}^x_d \rightarrow {\cal P}^a_d$ which are certain continuous parametrizations
of amplitudes vectors by the nodes. The exact definition of the maps $b_1,b_2$ will be given within
the proof of theorem \ref{thm.linear.subspaces}.

\smallskip

Denote by $\mathbbm{1} \in {\cal P}^a_d$ the amplitudes vector with all entries equal to $1$.

\begin{definition}
	For each $x \in {\cal P}^x_d$ the sets $P_1(x),P_2(x) \subset {\cal P}^a_d$ are defined via the mappings $b_1(x),\allowbreak b_2(x)$
	as follows:
	$$
		P_1(x)=\left\{ \lambda\left( \frac{1}{d} \mathbbm{1} + b_1(x) + u_1\right) \in  {\cal P}^a_d \ : \ \lambda \ne 0,
		u_1 \perp span\left(\mathbbm{1},b_1(x)\right) \right\}.
	$$
	$$
		P_2(x)=\left\{ \lambda\left( \frac{1}{d} \mathbbm{1} + b_2(x) + u_2\right)\in   {\cal P}^a_d \ : \ \lambda \ne 0, u_2
		\perp span\left(\mathbbm{1},b_2(x)\right) \right\}.
	$$
\end{definition}

\begin{remark}
	For a given nodes vector $x \in {\cal P}^x_d$, the set $P_1(x) \subset {\cal P}^a_d$ (and similarly $P_2(x)$) is a
	relatively ``simple'' ``punctured'' vector space of dimension $d-1$ given by
	$$span\left(\frac{1}{d} \mathbbm{1} + b_1(x),Ker\big(\mathbbm{1},b_1(x)\big) \right) \cap  {\cal P}^a_d$$
	minus the subspace $Ker\big(\mathbbm{1},b_1(x)\big) $, where $Ker\big(\mathbbm{1},b_1(x)\big)$ denotes
	the subspace of vectors perpendicular to $\mathbbm{1},b_1(x)$.
\end{remark}

\begin{theorem}\label{thm.linear.subspaces}
	For $F=(a,x) \in {\cal P}_d$ with $m_0(F)\ne 0$, $F \in \Sigma_{d,1,2}$ iff  
	the amplitudes vector of $F$, $a$, belongs to at least one of the sets $P_1(x),P_2(x) \subset {\cal P}^a_d$.
\end{theorem}
% 
% \begin{theorem}\label{thm.low.rank}
% 	The set of signals $\Sigma_{d,1,2} \subset {\cal P}_d$ is parametrized according to    
% 	the bundles $P_1,P_2$ and $Q$. That is, each $F=(a,x) \in \Sigma_{d,1,2}$ has its amplitudes vector $a$ parametrized via
% 	the nodes $x$ such that $a$ belongs to at least one of the sets $P_1(x),P_2(x)$ or $Q(x)$ and 
% 	for each $F=(a,x)$ that belongs to either of the sets $P_1(x),P_2(x),Q(x)$, we have that $F \in \Sigma_{d,1,2}$.       
% \end{theorem}	
% 
\begin{proof}
% 	Let $F=\in {\cal P}_d,\; F(x)=\sum_{i=1}^d a_i\delta(x-x_i)$, such that there exists single node signal
% 	$\tilde{F}(x)=a\delta(x-y_1)$ with $m_0(\tilde{F})=m_0(F),\; m_1(\tilde{F})= m_1(F),\; m_2(\tilde{F})= m_2(F)$.
% 	
% 	\smallskip 
% 	
% 	Case \rom{1}: $m_0(F)=0$ then $m_0(\tilde{F})=m_0(F)$ imply $y_1=0$ 
	Let $F=(a,x) \in \Sigma_{d,1,2}$, $x=(x_1,\ldots,x_d),\;a=(a_1,\ldots,a_d)$, such that $m_0(F) \ne 0$.
	By Theorem \ref{thm.quad.form} we have that this is equivalent to the amplitudes vector of $F$, $a$, being a zero of
	the quadratic form $a^TD(x)a$. We now analyse the zeros of $a^TD(x)a$.
	
	Consider the following notation which simplifies the presentation. For a nodes vector $x=(x_1,\allowbreak
	\ldots,x_d)$:
	\begin{itemize}
	  \item $\mu = \mu(x) = \frac{1}{d} \sum_{i=1}^d x_i$ is the mean of the nodes vector.
	  \item $\bar{x}=x-\mu(x) {\mathbbm 1}$ is the nodes vector centered to its mean $\mu(x)$.
	  \item Finally $||\cdot||$ denotes the euclidean norm. 
	\end{itemize}

	\begin{lemma}\label{lemma.zeroes}	
		For $x=(x_1,\ldots,x_d)\in {\mathbb R}^d, a=(a_1,\ldots,a_d)\in {\mathbb R}^d$, with $x_i,\; i=1,\ldots,d,$ distinct
		and $\sum_{i=1}^da_i \ne 0$, we have that
		$$
			a^TD(x)a =0
			\mbox{\; iff \;} a=(\sum_{i=1}^d a_i)(\frac{1}{d}\mathbbm{1} + \alpha_k(x) \bar{x} +
			u),\tab k\in \{0,1\},
		$$ 
	    where:
		\begin{align}\label{eq.alpha}
			\begin{split}
				\alpha_{1,2} & =\frac{\frac{\bar{x}^T D {\mathbbm 1}}{d}+ c_{1,2}}{2  ||\bar{x}||^4},\\
				 c_{1,2} &= \pm \sqrt{\left(\frac{\bar{x}^T D {\mathbbm 1}}{d}\right)^2
				+4 ||\bar{x}||^4 \frac{1}{d^2} \sum_{1\le i<j \le d}d_{i,j}^2}
			\end{split}
		\end{align}
		and $u$ is any vector that is orthogonal to ${\mathbbm 1}$ and $\bar{x}$.
	\end{lemma}
	\begin{proof}
		\begin{equation}\label{eq.A.factorization}
			D(x)=diag(x x^T){\mathbbm 1}^T +{\mathbbm 1} \; diag(x x^T)^T - 2 x x^T, 
		\end{equation}
		where $diag(x x^T)=(x_1^2,\ldots,x_d^2)$ taken as a column vector.
		Consider the orthogonal projection into the subspace $\{x |\; x^T \cdot {\mathbbm 1} =0\}$ given by 
		$I-\frac{\1 \1^T}{d}$, where $I$ is the $d \times d$ identity matrix.
		Using equation (\ref{eq.A.factorization}), by direct calculation we have that
		\begin{align}\label{eq.edm.projection.form}
			\frac{1}{2}\left(I-\frac{\1 \1^T}{d}\right)D\left(I-\frac{\1 \1^T}{d}\right) =
					-(x-\mu \1)(x^T-\mu \1^T) 
					&=-\bar{x}\bar{x}^T  .
		\end{align}
		Let $a \in {\mathbb R}^d$ with $\sum_{i=1}^d a_i \ne 0$. Then $a=(\sum_{i=1}^d a_i)(\frac{1}{d} \1+\alpha\bar{x} +
		u), $ for some $\alpha \in \mathbb {\mathbb R}$ and for some vector $u$ which is orthogonal to subspace spanned by the vectors $\1$ and $\bar{x}$.
		Then
		\begin{align*}
			\frac{1}{2} a^T D a &= \frac{1}{2}(\frac{1}{d}\1+\alpha\bar{x} + u)^T D
			(\frac{1}{d}\1+\alpha\bar{x}+ u) \\  
					&= \frac{1}{d^2}\frac{1}{2}\1^T D \1 + \frac{1}{2}\alpha^2\bar{x}^T D \bar{x} + \frac{\alpha}{d}\bar{x}^T D \1 \\
					&=\frac{1}{d^2} \sum_{1\le i<j \le d}d_{i,j}^2 - \alpha^2 ||\bar{x}||^4 + \frac{\alpha}{d}\bar{x}^T D \1 , 
		\end{align*}
		where for the penultimate equality we used equation (\ref{eq.edm.projection.form}) to get
		$\frac{1}{2}\alpha^2\bar{x}^T D \bar{x} = -\alpha^2 ||\bar{x}||^4$. 
	Then setting $\frac{1}{d^2} \sum_{1\le i<j \le d}d_{i,j}^2 - \alpha^2 ||\bar{x}||^4 + \frac{\alpha}{d}\bar{x}^T D \1 =
	0$ we get that $\alpha$ is as declared in (\ref{eq.alpha}). This concludes the proof of Lemma \ref{lemma.zeroes}.
	\end{proof}
% 	\begin{proof}[Poof of Proposition \ref{prop.h.square}]
% 		Let $a=(\sum_{i=1}^s a_i)(\frac{1}{s} e+\alpha_k\bar{z}+\beta \bar{z} + u), $ 
% 		\begin{align}
% 			h^2 \frac{1}{2}\frac{|a^TAa|}{|\sum_{i=1}^s a_i|} & = h^2 \left|\sum_{i=1}^s a_i \right| 
% 			\left|\frac{1}{2}(\frac{1}{s}e+\alpha_k \bar{z}+\beta \bar{z}+u)^T A (\frac{1}{s}e+\alpha_k \bar{z}+\beta\bar{z}+u)
% 			\right|
% 			\\
% 				&= h^2 \left|\sum_{i=1}^s a_i \right| 
% 			\left|\frac{1}{2}(\frac{1}{s}e+\alpha_k \bar{z}+\beta \bar{z})^T A (\frac{1}{s}e+\alpha_k \bar{z}+\beta\bar{z})
% 			\right|,
% 		\end{align}
% 		where we used equation (\ref{eq.edm.projection.form}) to cancel out factors involving multiplication of $A$ and $u$.
% 		Proceeding
% 		\begin{align}
% 			&=  h^2 \left|\sum_{i=1}^s a_i \right|  \left|\beta^2 \frac{1}{2}\bar{z}^T A \bar{z} +\beta (\frac{1}{s}e +\alpha_k \bar{z})^T A
% 			\bar{z} \right|\;,
% 		\end{align}
% 		where we used lemma \ref{lemma.zeroes} to cancel out $(\frac{1}{s}e +\alpha_k \bar{z})^T A (\frac{1}{s}e +\alpha_k
% 		\bar{z})$. Proceeding
% 		\begin{align}
% 			&=h^2 \left|\sum_{i=1}^s a_i \right|  \left|-\beta^2 s^2 Var^2(z)-2\alpha_k \beta s^2 Var^2(z) + \beta
% 			\frac{\bar{z}^T A e}{s} \right| \\
% 			&= h^2 \left|\sum_{i=1}^s a_i \right|  |\beta| \left|  \alpha_k2s^2 Var^2(z)-
% 			\frac{\bar{z}^T A e}{s} + \beta s^2 Var^2(z)  \right |\\
% 			&=h^2 \left|\sum_{i=1}^s a_i \right|  |\beta| \left| C_k
% 			 + \beta s^2 Var^2(z)  \right |								
% 		\end{align}
% 		where for the first equality we used equation (\ref{eq.edm.projection.form}).
% 	\end{proof}
	Now setting $b_1(x)=\alpha_1(x)\bar{x}$ and $b_2(x)=\alpha_2(x)\bar{x}$ concludes the proof of Theorem \ref{thm.linear.subspaces}.
\end{proof}

\section{Lower bounds}\label{sec.lower.bound}
% 
% Let a signal $F(x)=\sum_{j=1}^{d}a_j\delta(x-x_j) \in {\cal P}_d$ be given and let
In this section we derive lower bounds on the moments difference between a signal with $d$ nodes, $F\in{\cal P}_d$,  and a signal with 
strictly less nodes, $\tilde{F}$. 

\smallskip

We will consider only the first $2d-1$ consecutive moments of each signal. 
Accordingly, we denote by $\bar {\cal M}_d \cong {\mathbb R}^{2d-1}$ the restricted moment space consisting of $2d-1$ tuples $(\nu_0,\ldots,\nu_{2d-2})$, 
and for any signal $G$, we put $\bar m(G)=(m_0(G),\ldots,m_{2d-2}(G))\in \bar {\cal M}_d$.

\smallskip
The Hankel matrix introduced in equation \eqref{eq.hankel.definition} plays an important part in what follows.
Recall that for a moment vector $m=(m_0,\ldots,m_{2d-1}) \in {\cal M}_d$ we defined
\begin{equation}
	H_{d}(m)=
	\begin{bmatrix}
		m_{0} & m_{1} & m_{2} & \dots & m_{d-1}\\
		m_{1} & m_{2} & m_{3} & \dots & m_{d}\\
		\adots & \adots & \adots & \adots & \adots\\
		m_{d-1} & m_{d} & m_{d+1} & \dots & m_{2d-2}
	\end{bmatrix}.
\end{equation}  
Note that this matrix depends only in the first $2d-2$ entries of $m$.
% \footnote{  
% In order to reconstruct a $d$ nodes signal 
% from exact moments measurements at least $2d$ moments are needed. Now consider the moment $m_{2d-1}(F)$,
% the Prony method only uses it to construct the right hand side. Here it appears 
% in the definition of $H_{d}(m)$ only as to stress that at least this many moments are needed.    
% }

\smallskip 

For a signal $G=(a,x) \in {\cal P}_l,\; l\le d$, $x=(x_1,\ldots,x_l),\;a=(a_1,\ldots,a_l)$,  we denote by $H_{d}(G)$ the Hankel matrix as above, formed by the 
$2d-1$ initial moments of $G$.  
As it was stated in the introduction, the rank of $H_{d}(G)$ is equal to the number 
of nodes in $G$.
% 
% \begin{equation}\label{eq:Hankel.Matrix.5}
% V_d(X)=				
% \begin{bmatrix}
% 1  & 1 &... & 1     \\
% x_{1}  & x_2 &... & x_{d}       \\
% \reflectbox{$\ddots$}  & \reflectbox{$\ddots$} &     &        \\					
% x_1^{d-1}& x_2^{d-1}   & ...    & x_d^{d-1}.
% \end{bmatrix}
% %\tab d=0,1,...,	
% \end{equation}
This fact can be seen as follows. Let $D_l(a)=diag (a_1,\ldots,a_l)$ be the diagonal matrix with the entries $a_1,\ldots,a_l$.
Let $V^{0:d-1}_l(x)$ be the matrix formed by the first $d$ rows of the Vandermonde matrix \eqref{eq.vandermonde} with the nodes $x$. 
We have the next identity
\be\label{eq:factor.Hankel}
	H_d(G)=V_l^{0:d-1}(x)D_l(a)[V^{0:d-1}_l(x)]^T.
\ee

From (\ref{eq:factor.Hankel}) we conclude that  
the rank of $H_d(G)$ is equal to the number of the nodes in $G$, that is $l$.

Informally we can expect that the size of the minors of order $l$ of $H_d(F)$ measures the distance from $F$ to 
the set of the signals $\tilde{F}$ with at most $l-1$ nodes. The following definition, and Theorem \ref{thm.moment.dif} below make this observation rigorous.

\bd\label{def:minors}
For $l=1,\ldots,d$, define $\delta_l(F)$ as the maximum of the absolute values of all the $l$-minors of the moment Hankel matrix $H_d(F)$.
\ed	
To simplify the statement of our results we assume below that both the nodes $x_j$ and the amplitudes $a_j$ 
are bounded in absolute value by $1$, and denote by $\bar {\cal P}_d$ the corresponding part of the signal space  ${\cal P}_d$. 
However, for the approximating low-rank signal $\tilde{F}$, no such assumptions are made.

\begin{theorem}\label{thm.moment.dif}
Let a signal $F \in \bar {\cal P}_d$ be given. Then for each signal $\tilde{F} \in {{\cal P}}_{l-1}$, $l\le d$, we
have 
$$
	||\bar m(F)-\bar m(\tilde{F})|| \ge \min \ \left(1, \ \frac{\delta_l(F)}{\sqrt{2l-1} l^2 l!(d+1)^{l-1}}\right).
$$
\end{theorem}
\begin{proof}
	By our assumptions we have $|m_k(F)| \leq d, \ k=0,1,\ldots$. 
	In other words, $\bar m(F)\in Q_{d}\subset \bar {\cal M}_d$, 
	where $Q_{d}$ is the coordinate cube of radius $d$ centered at the origin of $\bar {\cal M}_d$. 
	We can assume also that $\bar m(\tilde{F})\in Q_{d+1}$, since otherwise $||\bar m(F)-\bar m(\tilde{F})|| \ge 1$. 
	Therefore we restrict the consideration to the cube $Q_{d+1}$.
	
	\smallskip
	
	Now we fix $l\in \{1,\ldots,d\}$, and let $\hat H_l(F)$ be the minor of $H_d(F)$ 
	for which the determinant $\Delta_l(F)=\det \hat M_l(F)$ attains, in absolute value, 
	the maximum $\delta_l(F)$. This determinant $\Delta_l(F)$, 
	considered as a function $\Delta_l(\nu)$ of the moments entering the minor $\hat H_l(F)$, 
	is a polynomial of degree $l$ in $\nu=(\nu_0,\ldots,\nu_{2d-2})\in \bar {\cal M}_d$. 
	On $Q_{d+1}$ this polynomial is bounded in absolute value by $l!(d+1)^l$, 
	being the sum of $l!$ products of the moments. 
	Applying to $\Delta_l(\nu)$ the classical Markov inequality (see e.g. \cite{rivlin2003introduction}) 
	(with an appropriate adaptation to the cubic domain $ Q_{d+1}$) 
	we have that 
	$$
		\max_{\nu \in Q_{d+1}}\left|\frac{\partial \Delta_l}{\partial \nu_i}(\nu)\right| \le \frac{l^2}{d+1} \max_{\nu \in Q_{d+1}}\left| \Delta(\nu) \right|.
	 $$
	
	We conclude that
	$$
		||grad \ \Delta_l(\nu) || \leq \zeta:= \sqrt{2l-1} l^2 l!(d+1)^{l-1}
	$$
	for each $\nu \in Q_{d+1}$.
	
	\smallskip
	
	Consequently, for any two points $\nu,\nu'\in Q_{d+1}$ we have
	$$
		|\Delta_l(\nu)-\Delta_l(\nu')|\leq \zeta ||\nu-\nu'||,
	$$
	or
	\be\label{eq:grad.bound}
		||\nu-\nu'||\ge \frac{1}{\zeta}|\Delta_l(\nu)-\Delta_l(\nu')|.
	\ee
	Now we notice that for any signal $\tilde{F}$ with at most $l-1$ nodes we have $\Delta_l(\bar m(\tilde{F}))=0$. 
	Indeed, this follows immediately from the fact that the rank of $H_d(\tilde{F})$ is at most $l-1$. Applying (\ref{eq:grad.bound}) to the points 
	$\bar m(F),\bar m(\tilde{F})$ we
	get 
	$$
	||\bar m(F)-\bar m(\tilde{F})||\ge \frac{1}{\zeta}|\Delta_l(\bar m(F))-\Delta_l(\bar m(\tilde{F}))|=\frac{\delta_l(F)}{\zeta}.
	$$
	This completes the proof of Theorem \ref{thm.moment.dif}.
\end{proof} 

\noindent {\bf Remark 1.} The inequalities provided by Theorem \ref{thm.moment.dif} for different $l$ are not completely 
independent from one another. Indeed, the assumption that $\delta_l(F)>0$ implies $\delta_{l'}(F)>0, l'<l.$ 
Via linear algebra one can get explicit lower bound in this direction. We plan to present these results separately.

\medskip

\noindent {\bf Remark 2.} The result of Theorem \ref{thm.moment.dif} can be improved as follows: 
the same lower bound on the difference of the moments of $F$ and $\tilde{F}$ remains valid as applied only to those moments which enter 
the minor $\hat H_l(F)$. The proof remains verbally the same.    

\medskip

In order to describe specific classes of signals $F \in \bar {\cal P}_d$ for which Theorem \ref{thm.moment.dif} works, 
we have to make explicit assumptions on the separation of the nodes $x$ of the signal $F$, 
and on the lower bound of the size of its amplitudes $a$:

By the assumptions, the nodes $x_1,\ldots,x_d$ of a signal $F \in \bar {\cal P}_d$ belong to the interval $I=[-1,1]$. 
Let us assume now that for a certain $\eta$ with $0<\eta\leq \frac{2}{d-1}$, 
the distance between the neighboring nodes $x_j,x_{j+1}, \ j=1,\ldots,d-1,$ is at least $\eta$. 
We also assume that for a certain positive $\gamma \le 1$ the amplitudes $a_1,\ldots,a_d$ satisfy $|a_j|\geq \gamma, \ j=1,\ldots,d$. 
We will call signals $F$ satisfying these conditions, $(\eta,\gamma)$-regular.

\begin{theorem}\label{thm.det}
Let a signal $F = (a,x) \in \bar {\cal P}_d$, $x=(x_1,\ldots,x_d),\;a=(a_1,\ldots,a_d)$, be $(\eta,\gamma)$-regular. Then
$$
\delta_d(F)\ge \prod_{i=1}^{d-1}(i!)^2\eta^{d(d-1)}\gamma^d.
$$
\end{theorem}
\begin{proof}
	We use factorization (\ref{eq:factor.Hankel}) of the moment Hankel matrix $H_d(F)$ :
	$$
		H_d(F)=V^{0:d-1}_d(x)D_d(a)[V^{0:d-1}_d(x)]^T,
	$$
	% with $V_d(X)$ the Vandermonde matrix on the nodes $X=(x_1,\ldots,x_d)$, and $D_d(A)=diag (a_1,\ldots,a_d)$ 
	% the diagonal matrix with the entries $a_1,\ldots,a_d.$ 
	Taking the determinant, we obtain
	$$
		\delta_d(F)=|\det H_d(F)|=\prod_{i=1}^d |a_i| \ (\det V^{0:d-1}_d(x))^2\geq \gamma^d (\det V^{0:d-1}_d(x))^2.
	$$
	For the determinant of the Vandermonde matrix we have
	$$
		|\det V^{0:d-1}_d(x)|= \prod_{i>j}|x_i-x_j|\ge \eta^{\frac{d(d-1)}{2}}\prod_{i>j}|i-j| =  \eta^{\frac{d(d-1)}{2}} \prod_{i=1}^{d-1}i!,
	$$
	since for the nodes $x_j$ of an $(\eta,m)$-regular signal $F$ we have $|x_i-x_j|\ge \eta |i-j|$. We conclude that
	$$
		\delta_d(F)\geq  \gamma^d  \eta^{d(d-1)}\prod_{i=1}^{d-1}(i!)^2 .
	$$
	This completes the proof of Theorem \ref{thm.det}.
\end{proof}

Now we can apply Theorem \ref{thm.moment.dif}, with $l=d$, and get a lower bound on the error of any low-rank approximation 
of the moments of an $(\eta,\gamma)$-regular signal $F \in \bar {\cal P}_d$:

\bc\label{cor.lower.bound}
Let a signal $F \in \bar {\cal P}_d$ be $(\eta,\gamma)$-regular. Then for each $\tilde{F} \in {{\cal P}}_{l},\ l < d$, we have
$$
	||\bar m(F)-\bar m(\tilde{F})|| \ge \theta:=\min \left(1,\frac{\eta^{d(d-1)} \gamma^d \prod_{i=1}^{d-1}(i!)^2 }{\sqrt{2d-1} d^2 d!(d+1)^{d-1}} \right).
$$
% Thus the moments $\bar{m}(F)=(m_0(F),\ldots,m_{2d-2}(F))\in \bar {\cal M}_d$ cannot be approximated better than to a distance 
% $\theta$ by the moments $\bar{m}(\tilde{F})=(m_0(\tilde{F}),\ldots,m_{2d-2}(\tilde{F}))\in \bar {\cal M}_d$ of any signal $\tilde{F}$ with less than $d$ nodes.
\ec

\noindent {\bf Remark 3.} As it was mentioned in Remark 1 above, the lower bound for $\delta_d(F)$ provided by 
Theorem \ref{thm.det} for $(\eta,\gamma)$-regular signals $F$, implies explicit lower bounds for each $\delta_l(F), l<d.$ 
We expect these bounds to contain, for smaller $l$, smaller powers of the parameter $\eta$. 
In the case of ``positive'' signals $F$ (i.e. for all the amplitudes $a_j$ positive), 
such improved bounds for the principal minors of $H_d(F)$ can be, presumably, obtained via the Silvester criterion.

\medskip

An important special case of the low-rank approximation problem is when the nodes of the signal $F$ near collide. 
(Our assumption of $(\eta, \gamma)$-regularity, essentially, excludes nodes near collisions.) 
A natural initial step in the study of signals with near-colliding nodes is to assume that the nodes form a cluster of a size $h\ll 1$, 
but inside the cluster the nodes are positioned in a relatively uniform way.

\bd\label{def:reg.cluster}
A signal $F$ is said to form an $(h,\eta,\gamma)$-regular cluster, if its nodes are obtained by an $h$-downscaling of an $(\eta,\gamma)$-regular signal 
$G \in {\cal \barP}_d$.
\ed

\bc\label{cor:main.2}
Let a signal $F \in \bar {\cal P}_d$ form an $(h,\eta,\gamma)$-regular cluster. Then for each $\tilde{F} \in {{\cal P}}_{l},\ l<d$, we have
$$
	||\bar{m}(F)-\bar{m}(\tilde{F})|| \ge \theta_h:=\min (h^{2d-2},\frac{\eta^{d(d-1)} \gamma^d h^{2d-2} \prod_{i=1}^{d-1}(i!)^2 }{\sqrt{2d-1} d^2 d!(d+1)^{d-1}}).
$$
% Thus the moments $\bar{m}(F)=(m_0(F),\ldots,m_{2d-2}(F))\in \bar {\cal M}_d$ cannot be approximated better than 
% to a distance $\theta_h$ by the moments $\bar{m}(\tilde{F})=(m_0(\tilde{F}),\ldots,m_{2d-2}(\tilde{F}))\in \bar {\cal M}_d$ of any signal $\tilde{F}$ with less than $d$ nodes.
\ec
\begin{proof}
	Under a scaling by $h$ the $k$-th moment of $F$ is multiplied by $h^k$. 
	So the difference in the $k$-th coordinate of $\bar{m}(F)$ and $\bar{m}(\tilde{F})$ is multiplied by $h^k \ge h^{2d-2}$.
\end{proof}

\noindent {\bf Remark 4.} If we could bound from below the difference in the lower-order moments of $F$ and $\tilde{F}$, 
it would provide a better asymptotic behavior in $h\to 0$ in the bound of Corollary \ref{cor:main.2}. 
	
\bibliographystyle{plain}
\bibliography{bib}{}

\end{document}